\documentclass[11pt]{amsart}
\usepackage{amssymb,amscd}
\usepackage[all,cmtip]{xy}
\usepackage{epic}
\usepackage{url}
\usepackage[english]{babel}
\usepackage{lscape}
\numberwithin{equation}{section}
\newtheorem{theorem}{Theorem}[section]
\newtheorem{lemma}[theorem]{Lemma}
\newtheorem{proposition}[theorem]{Proposition}
\newtheorem{corollary}[theorem]{Corollary}

\newtheorem{definition-lemma}[theorem]{Definition-Lemma}
\theoremstyle{definition}

\theoremstyle{remark}

\def\FF{{\mathbb F}}

\def\PP{{\mathbb P}}
\def\QQ{{\mathbb Q}}

\def\ZZ{{\mathbb Z}}

\begin{document}

\title[Some properties of the Zeta functions of an optimal tower.]
{On the Zeta Functions of an optimal tower of function fields over $\FF_4$}

\author{Alexey Zaytsev and Gary McGuire}

\address{UCD, CSI, Belfield, Dublin 4, Ireland}

\email{ Gary.McGuire@ucd.ie and Alexey.Zaytsev@ucd.ie}


\begin{abstract}
In this paper we derive a recursion for the zeta function of each function field in the second 
Garcia-Stichtenoth tower when $q=2$.
We obtain our recursion by applying a theorem of Kani and Rosen that gives information
about the decomposition of the Jacobians.
This enables us to compute the zeta functions explicitly of the first six function fields.
\end{abstract}
\maketitle

\begin{section}{Introduction}

In the article \cite{G-S2}  H.\ Stichtenoth and A.\ Garcia found an optimal tower of 
function fields over  $\FF_{q^2}$.
In this article we study the zeta function and the $p$-rank of this tower in the case $q=2$.
We apply the Kani-Rosen decomposition theorem to obtain a recursive
formula for the zeta functions.  At each level there is one new factor in the zeta function,
which we are able to compute up to the sixth level. 
Previously Shum \cite{Shum} has computed the zeta functions up to the fifth level.

It was shown by Kleiman that if $E/F$ is a finite extension of function fields,
then the L-polynomial (numerator of the zeta function)  of $E$ is divisible by the L-polynomial of $F$.
This result implies that the L-polynomial 
of the $n$-th level in the tower will divide the L-polynomial of the $(n+1)$-th level. 
Our results use the Kani-Rosen decomposition to 
give explicit formulae for the L-polynomial of each step
as a factor of the L-polynomial of a higher step.

Finally we prove that the $2$-rank of each level is maximal, showing that
the Jacobians of these function fields are ordinary abelian varieties.

The paper is organized as follows.
In Section \ref{backg} we give background.
Section \ref{kanirosen} gives the details of the Kani-Rosen decomposition in our situation.
Next in Section \ref{exs} we compute the L-polynomial for $n\leq 6$.
We generalize the examples in Section \ref{zeta}
to give some general results about the factorization of the L-polynomial.
Section \ref{ordinary} proves that the tower is ordinary.

\end{section}

\begin{section}{Background on the Tower}\label{backg}

We start with a review of some facts about the tower.

Let ${\rm T}_1:=\FF_{4}(x_1)$ be a rational function field. 
The second Garcia-Stichtenoth tower is defined by
$$
{\rm T}_{n}:=\FF_{4}(x_1, \ldots, x_n), \, 
$$
where
$$
x_{i}^2+x_{i}=\frac{x_{i-1}^{3}}{x_{i-1}^{2}+x_{i-1}}.
$$
We will always let $C_n$ denote a curve with function field $T_n$.

This tower is an optimal tower, i.e.,
$$
{\rm lim}\frac{N({\rm T}_n)}{g({\rm T}_n)}=\sqrt{4}-1=1 \quad \mbox{as} \quad n \rightarrow \infty.
$$
The genus is given by  
$$
g({\rm T}_n)= 
\left\{
\begin{array}{c}
(2^{n/2}-1)^2  \qquad \mbox{if}\, n   \, \mbox{is even}\\
(2^{(n+1)/2}-1)(2^{(n-1)/2}-1)\, \mbox{otherwise.}\\
\end{array}
\right.
$$

An important ingredient of this paper is a theorem of Kani and Rosen \cite{KR}, 
which follows work of Accola.
Let $G$ be contained in the automorphism group of a curve $C$.
Let ${\rm Jac}(C)$ be the Jacobian of $C$.
The Kani-Rosen theorem concerns isogenies and idempotents in the group algebra $\QQ[G]$.  
 For any subgroup $H$ of $G$
there is an idempotent
\[
\varepsilon_H = \frac{1}{|H|}\sum_{h\in H} h.
\]
If $G$ is the Klein 4-group with subgroups $H_1, H_2, H_3$, we have the relation
\[
\varepsilon_1 + 2\varepsilon_G=\varepsilon_{H_1}+\varepsilon_{H_2}+\varepsilon_{H_3}.
\]
 Applying the Kani-Rosen theorem we get an isogeny
 \[
 {\rm Jac}(C) \times {\rm Jac}({C/G})^2 \sim {\rm Jac}({C/H_1})\times {\rm Jac}({C/H_2})\times {\rm Jac}({C/H_3}).
 \]
 We will use this isogeny repeatedly.
 In order to use it, we need the following proposition.

\begin{proposition}\label{galois} If $n \ge 3$, then
the extension ${\rm T}_{n}$ over ${\rm T}_{n-2}$ is Galois 
and 
$${\rm Gal}({\rm T}_{n}/{\rm T}_{n-2}) \cong \ZZ/2\ZZ \times \ZZ/2\ZZ.$$ 
\end{proposition}
\begin{proof}
It suffices to prove the case $n=3$. Let $\tilde{\rm T}_3$ denote the Galois closure of ${\rm T}_3$ over ${\rm T}_1$.
Then $\tilde{\rm T}_3={\rm T}_3(u)={\rm T}_3(w)$, where $u$ and $w$ are roots of polynomials
$X^2+X+(x_2+1)^3/(x_2^2+x_2)$ and $X^2+X+1$, respectively  (see \cite{Zaytsev1}). Therefore we have that $\tilde{\rm T}_3={\rm T}_3$.

Now we can describe the Galois group explicitly:
\begin{enumerate}
\item {$\sigma_{0,0}$ is the identity map,}
\item{$\sigma: {\rm T}_3 \rightarrow {\rm T}_3$, such that   $x_2 \mapsto x_2$ and $x_3 \mapsto x_3+1$,}
\item{$\tau: {\rm T}_3 \rightarrow {\rm T}_3$, such that $x_2 \mapsto x_2+1$ 
and $x_3 \mapsto x_3+1/x_1+\gamma$, where $\gamma \in \FF_{4}$ such that $\gamma^2+\gamma+1=0$ (this $\gamma$ is fixed for the entire paper),}
\item{$\sigma \tau: {\rm T}_3 \rightarrow {\rm T}_3$,  such that $x_2 \mapsto x_2+1$ 
and $x_3 \mapsto x_3+1/x_1+\gamma+1$, with the same  $\gamma \in \FF_{4}$ .}
\end{enumerate}

Moreover, we know that ${\rm Gal({\rm T}_3/{\rm T}_1)}$ has order $4$ and hence it is isomorphic to either a cyclic group or $\ZZ/2\ZZ \times \ZZ/2\ZZ$. One can check that it is the second case.

In general the situation is similar. 
\end{proof}

It is not hard to obtain an equation for the Galois extension ${\rm T}_3/{\rm T}_1$. 
In fact, ${\rm T}_3=\FF_{4}(x_1,x_3)$ and a minimal polynomial 
of $x_3$ over $\FF_{4}(x_1)$ has the following form:
$$
T^4+\left(\frac{1}{x_1^2}+\frac{1}{x_1}\right)T^2+\left(\frac{1}{x_1^2}+\frac{1}{x_1}+1\right)T+\left(\frac{x_1^2}{x_1+1}\right)^2.
$$
\end{section}

\begin{section}{Decomposition of Jacobians}\label{kanirosen}

The fact that the extension ${\rm T}_{n+2}$ over ${\rm T}_{n}$ is Galois (see Prop.\ \ref{galois}) provides a decomposition of the Jacobians of the corresponding curves, by the Kani-Rosen theorem.
In this section we give the details of the decomposition.

Let $C_n$ be a curve with the function field ${\rm T}_{n}$. 
Then the Galois covering $C_{n} \rightarrow C_{n-2}$  implies a Kani-Rosen decomposition \cite{KR} 
of the Jacobian of the curve $C_{n}$, as we outlined in the previous section.
The Galois automorphism group  
$\langle \sigma,\,  \tau \rangle$ is isomorphic to the Klein 4-group,
and hence we have the following diagram of coverings
$$
\xymatrix{
 &\ar[ld]_{2:1}C_{n}\ar[d]^{2:1}\ar[rd]^{2:1}&\\
 C_{n-1}\cong C_{n}/\langle \sigma\rangle \ar[rd]^{2:1}& C_{n}/\langle \sigma \tau  \rangle \ar[d]^{2:1} &C_{n}/\langle \tau  \rangle  \ar[ld]_{2:1} \\
&C_{n-2} \cong C_{n}/\langle \sigma, \, \tau \rangle&\\}
$$
and the following isogeny of Jacobians
 \begin{equation}\label{jacdecomp}
{\rm Jac}(C_n) \times {\rm Jac}(C_{n-2})^2 
\sim {\rm Jac}(C_{n-1}) \times
{\rm Jac}(C_{n}/\langle \sigma \tau \rangle) \times
 {\rm Jac}(C_{n}/\langle \tau  \rangle).
\end{equation}

Therefore if we denote the ${\rm L}$-polynomial of the curve $C$ by ${\rm L}_{C}(X)$
 then we have the following decomposition
\begin{equation} \label{Frobrelation}
 {\rm L}_{C_n}(T)\ {\rm L}_{C_{n-2}}(T)^2 =
 {\rm L}_{C_{n-1}}(T)\
 {\rm L}_{C_{n}/\langle \sigma \tau \rangle}(T)\
 {\rm L}_{C_{n}/\langle \tau \rangle}(T).
\end{equation}

\bigskip

Moreover, one can see that the function fields are
$$\FF_{4}(C_{n}/\langle  \tau \rangle)=\FF_{4}(x_1, \ldots, x_{n-2},y_0)$$  
where $y_0$ is a root of 
$$X^2+(1/x_{n-2}+\gamma)X+\frac{x_{n-2}^2}{x_{n-2}+1},$$ 
and 
$$\FF_{4}(C_{n}/\langle  \sigma \tau \rangle)=\FF_{4}(x_1, \ldots, x_{n-2},y_1)$$
and  $y_1$ is a root of 
$$X^2+(1/x_{n-2}+\gamma+1)X+\frac{x_{n-2}^2}{x_{n-2}+1},$$ 
with $\gamma \in \FF_4$ such that $\gamma^2+\gamma+1=0$.

The function fields can also be written in terms of Artin-Scheier generators.
In fact, the substitutions $X=(1/x_{n-2}+\gamma)T$ and $X=(1/x_{n-2}+\gamma+1)T$ yield that
$$\FF_{4}(x_1,x_2, \ldots,x_{n-2},y_0)=\FF_{4}(x_1,x_2, \ldots,x_{n-2},u_0),$$
and 
$$\FF_{4}(x_1,x_2, \ldots,x_{n-2},y_1)=\FF_{4}(x_1,x_2, \ldots,x_{n-2},u_1),$$
where $u_0$ and $u_1$ are roots of the polynomials
$$
T^2+T+\left( \frac{x_{n-2}^2}{x_{n-2}+1} \right) \left(\frac{x_{n-2}^2}{1+\gamma^2 x_{n-2}^2}\right)
$$
and 
$$
T^2+T+\left( \frac{x_{n-2}^2}{x_{n-2}+1} \right) \left(\frac{x_{n-2}^2}{1+\gamma^2 x_{n-2}^2 +x_{n-2}^2}\right),
$$
respectively.

We finish  with a useful observation. 
\begin{lemma}\label{usefullemma}
The function fields
$\FF_{4}(x_1,x_2, \ldots,x_{n-2},u_0)$ and $\FF_{4}(x_1,x_2, \ldots,x_{n-2},u_1)$
have the same zeta function.
\end{lemma}

Proof:
We show that the two fields are isomorphic over $\FF_2$ by the map which sends
$\gamma \mapsto \gamma+1$ and fixes each $x_i$.
Indeed, this map sends the minimal polynomial of $u_0$ to
the minimal polynomial of $u_1$.
This map preserves the numbers of places of each degree (over $\FF_4$), and hence
these two function fields have the same zeta function.

\bigskip


\begin{section}{The ${\rm L}$-polynomial of ${\rm T}_n$ for $n\leq 6$.}\label{exs}

In principle, our method allows us to recursively compute 
the L-polynomials ${\rm L}_{{\rm T}_n} (T)$ from the L-polynomials
of lower levels in the tower, and the L-polynomials of the other quotients.
In practice one of the quotients is hard to compute. 
We are able to compute it up to $n=6$, which we do in this section.

\begin{subsection}{The $n=3$ Case}

Here we  compute the ${\rm L}$-polynomial of ${\rm T}_3$.
(Of course this small case can also be done in other ways.)

The function fields $\FF_{4}(C_{3}/\langle \sigma \tau \rangle)$ and $\FF_{4}(C_{3}/\langle \tau \rangle)$ can be written explicitly, namely
$$
\FF_{4}(C_{3}/\langle \sigma \tau \rangle)=\FF_{4}(x_1, x_3(x_3+1/x_1+\gamma+1))
$$
and
$$
\FF_{4}(C_{3}/\langle \tau \rangle)=\FF_{4}(x_1, x_3(x_3+1/x_1+\gamma)).
$$
If we denote $x_3(x_3+1/x_1+\gamma+1)$ by $y_0$ and $x_3(x_3+1/x_1+\gamma)$ by $y_1$, then the minimal polynomials of $y_0$ and $y_1$ are
$$
X^2+(1/x_1+\gamma+1)X+\frac{x_1^2}{x_1+1},
$$
and 
$$
X^2+(1/x_1+\gamma)X+\frac{x_1^2}{x_1+1},
$$
respectively.

The function fields  can also be rewritten  in terms of Artin-Schreier extension, as we outlined above.
Indeed, the substitution $X=(1/x_1+\gamma)T$ yields
$\FF_{4}(x_1,y_0)=\FF_{4}(x_1, u_0)$, where $u_0$ is a root of the polynomial
$$
T^2+T+\left( \frac{x_1^2}{x_1+1} \right) \left(\frac{x_1^2}{1+\gamma^2x_1^2}\right).
$$

The following ${\rm L}$-polynomials can be easily computed:
$$
\begin{array}{l}
{\rm L}_{{\rm T}_1}(T)={\rm L}_{\PP^1}=1,\\
{ {\rm L}_{C_{3}/\langle \sigma \rangle} (T)={\rm L}_{{\rm T}_2} (T)=1+3T+4T^2, }\\
{ {\rm L}_{C_{3}/\langle \sigma \tau \rangle} (T)=1+3T+4T^2, }\\
{ {\rm L}_{C_{3}/\langle \tau \rangle} (T)=1+3T+4T^2. }\\
\end{array}
$$
Therefore, by  (\ref{Frobrelation}) we get
$$
{\rm L}_{{\rm T}_3} (T)=(1+3T+4T^2)^3.
$$

\end{subsection}

\subsection{The $n=4$ Case}

Here we  find the ${\rm L}$-polynomial of ${\rm T}_4$ and the other important subfields.

Similar to  the previous section, from (\ref{jacdecomp}) 
we have the following decomposition into isogeny factors
$$
{\rm Jac}({\rm T}_{2})^2 \times {\rm Jac}({\rm T}_4) 
\sim {\rm Jac}(C_{3}) \times
{\rm Jac}(C_{4}/\langle \sigma \tau \rangle) \times
 {\rm Jac}(C_{4}/\langle \tau  \rangle),
$$
where
$$
\FF_{4}(C_{4}/\langle \sigma \tau \rangle)=\FF_{4}(x_1,x_2,  x_4(x_4+1/x_2+\gamma+1))
=\FF_{4}(x_1,x_2,  u_0),
$$
$$
\FF_{4}(C_{4}/\langle \tau \rangle)=\FF_{4}(x_1, x_2,  x_4(x_4+1/x_2+\gamma))=\FF_{4}(x_1,x_2,  u_1),
$$
where  $u_0$ and $u_1$ are roots of polynomials
$$T^2+T+\left( \frac{x_2^2}{x_2+1} \right) \left(\frac{x_2^2}{1+\gamma^2x_2^2}\right)$$
and
$$T^2+T+\left( \frac{x_2^2}{x_2+1} \right) \left(\frac{x_2^2}{1+\gamma^2x_2^2+x_2^2}\right),$$
respectively.

Then we have the following diagram 

$$
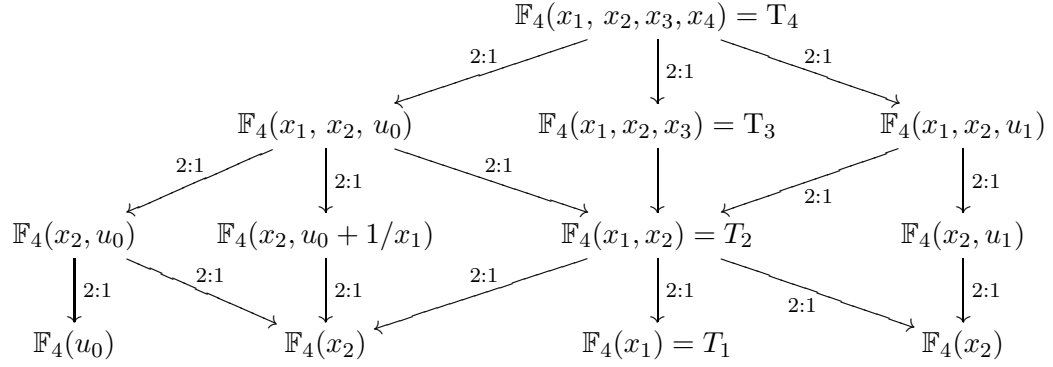
\xymatrix{
&&\ar[ld]_{2:1} \FF_{4}(x_1,\, x_2,x_3, x_4)={\rm T}_4 \ar[d]^{2:1}\ar[rd]^{2:1}&&&\\
&\ar[ld]_{2:1} \FF_{4}(x_1,\, x_2,\, u_0) \ar[d]^{2:1}\ar[rd]^{2:1}&
\FF_{4}(x_1,x_2,x_3)={\rm T}_3 \ar[d]&\FF_{4}(x_1,x_2,u_1)  \ar[d]^{2:1} \ar[ld]^{2:1}&&\\
\FF_{4}(x_2,u_0)  \ar[rd]^{2:1}  \ar[d]^{2:1}&\FF_{4}(x_2, u_0+1/x_1)  \ar[d]^{2:1} &  \ar[ld]_{2:1} \FF_{4}(x_1,x_2) =T_2\ar[rd]_{2:1} \ar[d]^{2:1} &\FF_{4}(x_2,u_1) \ar[d]^{2:1}&&\\
\FF_{4}(u_0) &\FF_{4}(x_2)& \FF_{4}(x_1)=T_1&\FF_{4}(x_2)&&\\
&&& &&\\}
$$
From this diagram it follows that the extension $\FF_{4}(x_1, x_2, u_0)$ over 
$\FF_{4}(x_2)$ is a Galois cover with a Galois group isomorphic
to $\ZZ/2\ZZ \times \ZZ/2\ZZ$.
We then apply (\ref{Frobrelation}) to get
$$
{\rm L}_{\FF_{4}(x_2)}^2 (T) \, {\rm L}_{\FF_{4}(x_1,x_2,u_0)}(T)=
 {\rm L}_{\FF_{4}(x_2,u_0)}(T)\,  
 {\rm L}_{\FF_{4}(x_1,x_2)}(T)\,  {\rm L}_{\FF_{4}(x_2, u_0+1/x_1)}(T).
$$
Here we also would like to remark that
$$
(u_0+1/x_1)^2+(u_0+1/x_1)=\frac{1}{x_2^2+x_2}+\left( \frac{x_2^2}{x_2+1} \right) \left(\frac{x_2^2}{1+\gamma^2x_2^2}\right).
$$

In order to use  (\ref{Frobrelation}) we require
$$
\begin{array}{l}
{\rm L}_{{\rm T}_2}(T)=1+3T+4T^2 \\
{\rm L}_{C_{4}/\langle \sigma \rangle}(T)={\rm L}_{{\rm T}_3}(T)=(1+3T+4T^2)^3,\\
{\rm L}_{C_{4}/\langle \tau \rangle}(T)=(1+3T+4T^2)^3(1-T+4T^2),\\
{\rm L}_{C_{4}/\langle \sigma \tau \rangle}(T)=(1+3T+4T^2)^3(1-T+4T^2).\\
\end{array}
$$
The first two of these were obtained in the previous section, and the
second two were computed using Magma \cite{Magma}.
The second two are necessarily equal by Lemma \ref{usefullemma}.
As a result, from (\ref{Frobrelation})  we get 
$$
{\rm L}_{{\rm T}_4}=(1-T+4T^2)^2(1+3T+4T^2)^7.\\
$$
\end{section}

\begin{subsection}{The $n=5$ Case}

In this section we  continue our recursive construction to level 5 of the tower.

The decomposition of Jacobian into isogeny factors has the form
$$
{\rm Jac}({\rm T}_{3})^2 \times {\rm Jac}({\rm T}_5)
\sim {\rm Jac}(C_{4}) \times
{\rm Jac}(C_{5}/\langle \sigma \tau \rangle) \times
{\rm Jac}(C_{5}/\langle \tau  \rangle),
$$
where

$$
\FF_{4}(C_{5}/\langle \sigma \tau \rangle)=\FF_{4}(x_1,x_2, x_3,  u_0),
$$
$$
\FF_{4}(C_{5}/\langle \tau \rangle)=\FF_{4}(x_1, x_2, x_3,  u_1).
$$
where $u_0$ and $u_1$ are roots of polynomials
$$
T^2+T+\left( \frac{x_{3}^2}{x_{3}+1} \right) \left(\frac{x_{3}^2}{1+\gamma^2 x_{3}^2}\right),
$$
and
$$
T^2+T+\left( \frac{x_{3}^2}{x_{3}+1} \right) \left(\frac{x_{3}^2}{1+\gamma^2 x_{3}^2+x_{3}^2}\right),
$$
respectively.


Therefore we have the following diagram of extensions of degree $2$:
$$
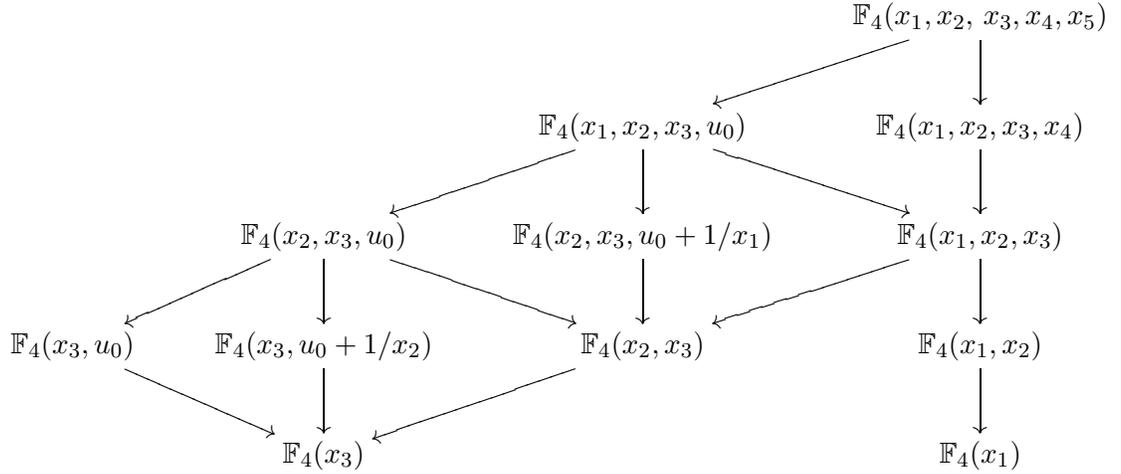
\xymatrix{
&&&\ar[ld] \FF_{4}(x_1, x_2,\, x_3,x_4,x_5) \ar[d]&&\\
&&\ar[ld] \FF_{4}( x_1, x_2,x_3,u_0) \ar[d]  \ar[rd] &\FF_{4}(x_1,x_2,x_3,x_4) \ar[d] &  \\
&\FF_{4}(x_2,x_3,u_0)\ar[rd] \ar[d] \ar[ld]&\FF_{4}(x_2,x_3, u_0+1/x_1) \ar[d] & \ar[ld] \FF_{4}(x_1,x_2,x_3) \ar[d] &&\\
\FF_{4}(x_3,u_0)\ar[rd]&\FF_{4}(x_3, u_0+1/x_2)\ar[d]&\ar[ld]\FF_{4}(x_2,x_3)& \FF_{4}(x_1,x_2)\ar[d]&\\
&\FF_{4}(x_3)&&\FF_{4}(x_1)&}
$$

Here we would like to remark that $u_0+1/x_1$ and $u_0+1/x_2$ are roots of polynomials
$$
T^2+T+\frac{1}{x_2^2+x_2}+\left( \frac{x_3^2}{x_3+1} \right) \left(\frac{x_3^2}{1+\gamma^2x_3^2}\right)
$$
and
$$
T^2+T+\frac{1}{x_3^2+x_3}+\left( \frac{x_3^2}{x_3+1} \right) \left(\frac{x_3^2}{1+\gamma^2x_3^2}\right),
$$
respectively.

From earlier calculations we know
$$
\begin{array}{l}
{\rm L}_{{\rm T}_3}(T)=(1+3T+4T^4)^3\\
{\rm L}_{C_{5}/\langle \sigma \rangle}(T)={\rm L}_{{\rm T}_4}(T)=(1-T+4T^2)^2(1+3T+4T^2)^7,\\
\end{array}
$$
and using Magma we get:
$$
\begin{array}{l}
{\rm L}_{C_{5}/\langle \sigma \tau \rangle}(T)=(1-T+4T^2)(1+T+4T^2)(1+3T+4T^2)^5(1+2T+T^2+8T^3+16T^4),\\
{\rm L}_{C_{5}/\langle \tau \rangle}(T)=(1-T+4T^2)(1+T+4T^2)(1+3T+4T^2)^5(1+2T+T^2+8T^3+16T^4).\\
\end{array}
$$
We note that these two L-polynomials are equal, which follows from Lemma \ref{usefullemma}.
At any rate, (\ref{Frobrelation}) gives
$$
{\rm L}_{{\rm T}_5}(T)=(1-T+4T^2)^4(1+3T+4T^2)^{11}(1+T+4T^2)^2(1+2T+T^2+8T^3+16T^4)^2.
$$

\end{subsection}

\begin{subsection}{The $n=6$ Case} 

The last part of this section is devoted to the computation 
of  the ${\rm L}$-polynomial of the function field ${\rm T}_6$.
\bigskip

In the usual way, using  (\ref{Frobrelation})  and Lemma \ref{usefullemma}   we have
$$
{\rm  L}_{\FF_{4}(x_1,x_2,x_3,x_4,x_5,x_6)}(T)\ {\rm L}_{\FF_{4}(x_1, x_2,x_3,x_4)} (T)^2 =
{\rm L}_{\FF_{4}(x_1, x_2, x_3,x_4, u_0)} (T)^2\  {\rm L}_{\FF_{4}(x_1,x_2,x_3,x_4,x_5)} (T).
$$
Therefore we need ${\rm L}_{\FF_{4}(x_1,x_2, x_3,x_4, u_0)} (T)$ and then we are done.

Applying  the Kani-Rosen decomposition again (see diagram of degree two extensions below)
and Lemma \ref{usefullemma}
we obtain the following equalities:
 
$$
{\rm  L}_{\FF_{4}(x_1,x_2,x_3,x_4,u_0)}(T)\  {\rm L}_{\FF_{4}( x_2,x_3,x_4)}  (T)^2=
{\rm L}_{\FF_{4}(x_2, x_3,x_4,u_0+1/x_1)} (T)\ {\rm L}_{\FF_{4}(x_2,x_3,x_4,u_0)}(T)^2,  
$$
and
$$
{\rm  L}_{\FF_{4}(x_2,x_3,x_4,u_0)}(T) {\rm L}_{\FF_{4}( x_3,x_4)} (T)^2 =
{\rm L}_{\FF_{4}(x_3,x_4,u_0+1/x_2)} (T) {\rm L}_{\FF_{4}(x_3,x_4,u_0)}(T)
{\rm L}_{\FF_{4}(x_2,x_3,x_4)}(T).  
$$

We need the following two new ${\rm L}$-polynomials (which were again computed in Magma)
$$
\begin{array}{rl}
L_{\FF_{4}(x_2,x_3,x_4, u+1/x_1)}=&(1+T+4T^2)^2(1+3T+4T^2)^4(4+2T+T^2+8T^3+16T^4)\\
&(1+T+T^2+3T^3+4T^4+16T^5+64T^6),\\
L_{\FF_{4}(x_2,x_3,x_4, u_0)}=&(1+T+4T^2)^3(1+3T+4T^2)^{10} (4+2T+T^2+8T^3+16T^4)^2\\
&(1-T+4T^2)^3(1+T-T^2+3T^3-4T^4+16T^5+64T^6).\\
\end{array}
$$
Using these and the decomposition formulas above,
one can derive the ${\rm L}$-polynomial of ${\rm T}_6$:
$$
\begin{array}{rl}
L_{{\rm T}_6}=&(1+3T+4T^2)^{17}(1-T+4T^2)^{6}(1+T+4T^2)^{2}(4+2T+T^2+8T^3+16T^4)^{6}\\
&(1+T-T^2+3T^3-4T^4+16T^5+64T^6)^2.
\end{array}
$$

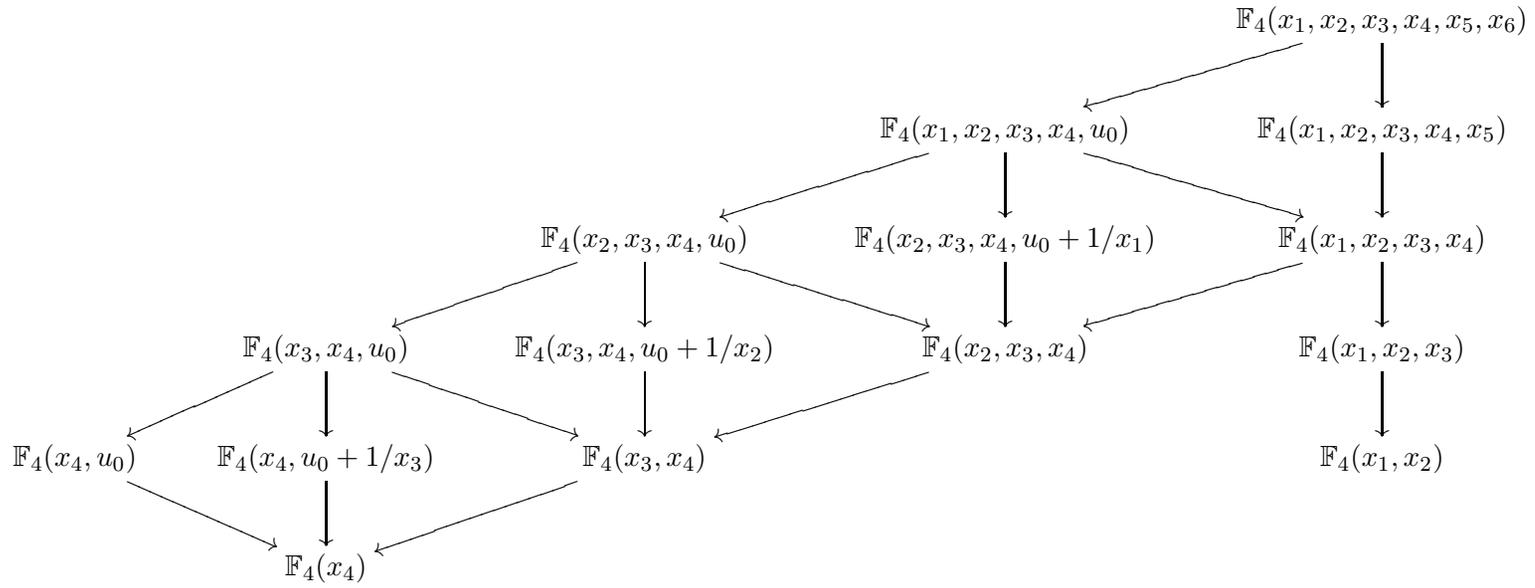
\begin{landscape}
$$
\xymatrix{
&&&&\ar[ld] \FF_{4}(x_1, x_2, x_3,x_4,x_5,x_6) \ar[d]&&\\
&&&\ar[ld] \FF_{4}( x_1, x_2,x_3,x_4,u_0) \ar[d]  \ar[rd] &\FF_{4}(x_1,x_2,x_3,x_4,x_5) \ar[d] &  \\
&&\FF_{4}(x_2,x_3,x_4,u_0)\ar[rd] \ar[d] \ar[ld]&\FF_{4}(x_2,x_3,x_4, u_0+1/x_1) \ar[d] & \ar[ld] \FF_{4}(x_1,x_2,x_3,x_4) \ar[d] &&\\
&\FF_{4}(x_3,x_4,u_0)\ar[rd] \ar[ld] \ar[d]&\FF_{4}(x_3,x_4, u_0+1/x_2)\ar[d]&\ar[ld]\FF_{4}(x_2,x_3,x_4)& \FF_{4}(x_1,x_2,x_3)\ar[d]&\\
\FF_{4}(x_4,u_0)\ar[rd] &\FF_4(x_4, u_0+1/x_3) \ar[d]& \ar[ld] \FF_{4}(x_3,x_4)&&\FF_{4}(x_1,x_2)&\\
&\FF_{4}(x_4)&&&&&}
$$
\end{landscape}

\end{subsection}

\end{section}

\begin{section}{Recurrence relations and the general zeta function}\label{zeta}

In this section we produce a general recurrence
formula for the ${\rm L}$-polynomials of the tower
and compute the degree of the unknown factor. 

We start with description of general situation.
As in the previous section we introduce new function fields and their Picard groups.
Recall that, if $T$ is a function field,  ${\rm Pic^{0}}({\rm T})$ is isomorphic to the Jacobian
of the curve corresponding to $T$.

Let $u$ be a root of a polynomial 
$
T^2+T+\displaystyle \frac{x_n^{4}}{(x_n+1)(1+\gamma^2x_n^2)},
$
so $u$ is like $u_0$ in previous sections.
Then we set

\begin{tabular}{l}
$F_{n}:=\FF_{4}(x_1, \dots, x_n, u)$, $X_{n}:={\rm Pic}^{0}(\FF_{4}(x_1, \dots, x_n, u))$, \\
$F_{n,1}:=\FF_{4}(x_2, \ldots, x_{n}, u+1/x_{1})$, $X_{n,1}:={\rm Pic^{0}}(\FF_{4}(x_2, \ldots, x_{n}, u+1/x_{1}))$,\\
${\rm J}_{n}:={\rm Pic^{0}}({\rm T}_{n}) \cong {\rm Jac}(C_n)$.\\
\end{tabular}

Due to the recursive nature of the tower and the isomorphisms of corresponding function fields, 
we get the following isomorphisms and  isogenies of abelian varieties
for any $m\geq 1$:
\begin{enumerate}
\item{$X_{n} \cong {\rm Pic^0}(\FF_{4}(x_{m}, \ldots, x_{m+n-1}, w))$, where $w$ is a root of  a polynomial  $T^2+T+\displaystyle \frac{x_{n+m-1}^{4}}{(x_{n+m-1}+1)(1+\gamma^2x_n^2)}$,}
\item{$X_{n} \sim {\rm Pic^0}(\FF_{4}(x_{m}, \ldots, x_{m+n-1}, t))$, where $t$ is a root of  a polynomial  $T^2+T+\displaystyle \frac{x_{n+m-1}^{4}}{(x_{n+m-1}+1)(1+(\gamma+1)^2x_n^2)}$,}
\item{$X_{n,1}\cong {\rm Pic^0}(\FF_{4}(x_{m+1}, \ldots, x_{m+n-1}, w+1/x_{m})) $,}
\item{$X_{n,1}\sim {\rm Pic^0}(\FF_{4}(x_{m+1}, \ldots, x_{m+n-1}, t+1/x_{m}))$,}
\item{${\rm J}_{n} \cong {\rm Pic^{0}}(\FF_{4}(x_{m}, \ldots, x_{m+n-1} )),$  with $n \ge 1$.}
\item{Owing to an inclusion $\FF_{4}(x_2, \ldots, x_{n}) \subset \FF_{4}(x_2, \ldots, x_{n}, u+1/x_{1})$ there exist an isogeny $X_{n,1} \sim {\rm J}_{n-1}\times {\rm Y}_{n,1}$, where ${\rm Y}_{n,1}$ is an abelian variety.}
 \end{enumerate}

Moreover, we can observe that
up to isomorphism there is the following diagram of extensions of degree $2$.
$$
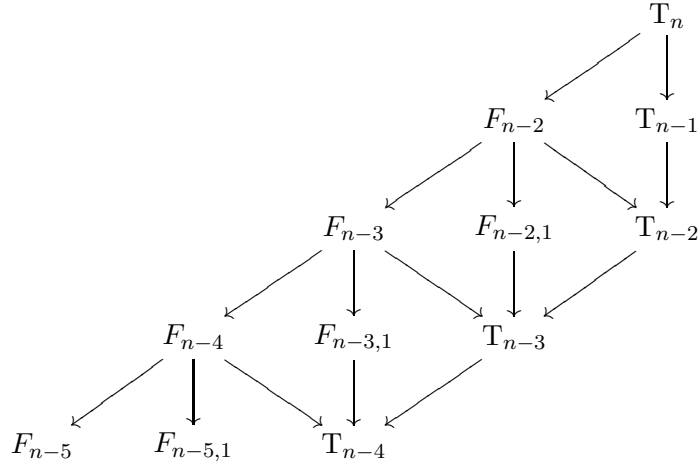
\xymatrix{
&&&&\ar[ld]{\rm T}_{n} \ar[d]\\
&&&\ar[ld] F_{n-2} \ar[d] \ar[rd]&{\rm T}_{n-1} \ar[d]\\
&&F_{n-3} \ar[d] \ar[ld] \ar[rd]&F_{n-2,1} \ar[d] &\ar[ld] {\rm T}_{n-2}\\
&F_{n-4} \ar[rd] \ar[ld] \ar[d]&F_{n-3,1}\ar[d]&{\rm T}_{n-3} \ar[ld]&\\
F_{n-5}&F_{n-5,1}&{\rm T}_{n-4}&&}
$$

\begin{proposition}
If $n \ge 3$ then there exists the following isogeny
$$
X_{n} \sim{\rm J}_{n} \times  X_1 \times X_{2,1} \times {\rm Y}_{3,1}\times \ldots \times {\rm Y}_{n,1}.
$$
\end{proposition}
\begin{proof}
We give a proof by induction.
The previous sections  provide the isogenies  
$$X_{2} \sim X_1 \times X_{2,1},$$
$$X_{3} \sim{J}_{3}\times  X_1 \times X_{2,1}\times {\rm Y}_{3,1},$$ and 
the induction hypothesis implies the isogeny
$$
X_{n-1} \sim{\rm J}_{n-1} \times  X_1 \times X_{2,1} \times {\rm Y}_{3,1}\times \ldots \times {\rm Y}_{n-1,1}.
$$

Based on the diagram of extensions and the Kani-Rosen decomposition, it follows that 
$X_{n}\times {\rm J}_{n-1}^2 \sim {\rm J}_{n}\times X_{n,1} \times X_{n-1}$. 
Recalling that
$X_{n,1} \sim {\rm Y}_{n,1} \times {\rm J}_{n-1}$
we get 
$$
X_{n} \sim {\rm J}_{n}\times Y_{n,1} \times  X_1 \times X_{2,1} \times {\rm Y}_{3,1}\times \ldots \times {\rm Y}_{n-1,1}.
$$
\end{proof}

The next theorem gives more information about the 
decomposition of ${\rm Pic}^{0}({\rm T}_n)$, and hence
 about the  decomposition of the ${\rm L}$-polynomial of ${\rm T}_{n}$.
\begin{theorem}
If $n \ge 5$ then
$$
{\rm J}_{n} \sim  {\rm J}_{n-1}\times X_1^2 \times X_{2,1}^2 
\times {\rm Y}_{3,1}^2\times \ldots \times {\rm Y}_{n-2,1}^2
$$
\end{theorem}
\begin{proof}
$$
{\rm J}_{n}\times {\rm J}_{n-2}^2 \sim X_{n-2}^2 \times {\rm J}_{n-1}
$$
and hence
$$
{\rm J}_{n} \sim  {\rm J}_{n-1}\times X_1^2 \times X_{2,1}^2 
\times {\rm Y}_{3,1}^2\times \ldots \times {\rm Y}_{n-2,1}^2.
$$
\end{proof}

\begin{corollary}
If $n \ge 5 $ then there is an isogeny
$$
{\rm J}_{n} \sim  \times X_1^{2n-3} \times X_{2,1}^{2(n-3)} 
\times {\rm Y}_{3,1}^{2(n-4)}\times \cdots \times {\rm Y}_{n-2,1}^2
$$

\end{corollary}

\begin{corollary}
The  ${\rm L}$-polynomial of the function field ${\rm T}_{n}$ has the following factorization
$$
{\rm L}_{{\rm T}_n}(T)=
{\rm L}_{X_1}(T)^{2n-3} \times {\rm L}_{X_{2,1}}(T)^{2n-6} 
\times {\rm L}_{{\rm Y}_{3,1}}(T)^{2n-8}\times \cdots \times{\rm L}_{{\rm Y}_{n-2,1}}(T)^2,
$$
or more precisely
$$
\begin{array}{l}
{\rm L}_{{\rm T}_n}(T)=(T^2+T+4)^{2n-8}(T^2+3T+4)^{12n-49}(T^2-T+4)^{6n-26}\\
(T^4+2T^3+T^2+8T+16)^{6n-24}(T^6+T^5-T^4+3T^3-4T^2+16T+64)^{2n-10}\\
{\rm L}_{Y_{5,1}}^{2n-12}\cdots {\rm L}_{Y_{n-2,1}}^2\\
\end{array}
$$
\end{corollary}

\begin{proof}

In  Section \ref{exs}
we found the $L-$polynomials of ${\rm T}_3$, ${\rm T}_4$, ${\rm T}_5$ and ${\rm T}_{6}$
and substituting those we get
$$
\begin{array}{l}
F_{{\rm T}_n}(T)=(T^2+3T+4)^{2n-3}((T^2+3T+4)(T^2-T+4)(T^4+2T^3+T^2+8T+16))^{2n-6}\\
((T^2+T+4)(T^4+2T^3+T^2+8T+16))^{2n-8}((T^2-T+4)^2(T^2+3T+4)^4\\
(T^4+2T^3+T^2+8T+16)(T^6+T^5-T^4+3T^3-4T^2+16T+64))^{2n-10}\\
F_{Y_{5,1}}^{2n-12}\ldots F_{Y_{n-2,1}}^2.
\end{array}
$$
\end{proof}

This provides some information about the order of the finite group
${\rm Pic}^{0}({\rm T}_{n})(\FF_{4})$.   The order is  
equal to ${\rm L}_{{\rm T}_{n}}(1)$.
From the previous corollary we may say
$$
\begin{array}{l}
\#{\rm Pic}^{0}({\rm T}_{n})(\FF_{4})=
{\rm L}_{{\rm T}_n}(1)=2^{58n-243}3^{2n-8}5^{2n-10} {\rm L}_{Y_{5,1}}^{2n-12}(1)\ldots {\rm L}_{Y_{n-2,1}}^2(1)
\end{array}
$$
and so we have information about the smoothness of the group order.
\bigskip


\subsection{Dimension of $Y_{n,1}$}

A crucial role in the computation of the Zeta function of ${\rm T}_{n+2}$ is 
the Zeta function of the factor $Y_{n,1}$. 
Therefore in order to estimate the indeterminateness 
we compute the dimension of this abelian variety,
and hence the degree of ${\rm L}-$polynomial of $Y_{n,1}$.

\begin{proposition}
$$
{\rm dim}(Y_{n,1})=\left\{
\begin{array}{ll}
2^{n-1},& \mbox{ \rm if}\quad n\quad \mbox{\rm is even}\\
2^{n-1}-2^{(n-3)/2}, &\mbox{\rm if}\quad n\quad \mbox{\rm is odd}.\\
\end{array}
\right.
$$
\end{proposition}
\begin{proof}
From the article of  H.\ Stichtenoth and A.\ Garcia \cite{G-S1} we know that
$$
g({\rm T}_n)= 
\left\{
\begin{array}{c}
(2^{n/2}-1)^2,  \quad \mbox{if}\quad n \quad \mbox{is even}\\
(2^{(n+1)/2}-1)(2^{(n-1)/2}-1),\quad \mbox{otherwise.}\\
\end{array}
\right.
$$

Moreover the isogenies of Jacobians yield that
\begin{itemize}
\item{
${\rm dim}(X_{n})=\frac{1}{2}({\rm dim}({\rm J}_{n+2})+2{\rm dim}({\rm J}_{n}) -{\rm dim}({\rm J}_{n+1})),$}
\item{${\rm dim}(X_{n,1})={\rm dim}(X_{n})+2{\rm dim}({\rm J}_{n-1})-{\rm dim}(X_{n-1})-{\rm dim}({\rm J}_{n}),$}
\item{${\rm dim}(Y_{n,1})={\rm dim}(X_{n,1})-{\rm dim}({\rm J}_{n-1}).$}
\end{itemize}

Combining all these facts together we get the desirable result.
\end{proof}

As a corollary we get the degree of 
the characteristic polynomial of Frobenius endomorphism
of an abelian variety $Y_{n,1}$.
\begin{corollary}
Let ${\rm L}_{Y_{n,1}}$ be the ${\rm L}$-polynomial 
of an abelian variety $Y_{n,1}$ (i.~e.\  it  is a reciprocal polynomial to the characteristic polynomial of Frobenius endomorphism). Then
$$
{\rm deg}({\rm L}_{Y_{n,1}})=\left\{
\begin{array}{ll}
2^{n},& \mbox{\rm if}\quad n\quad \mbox{\rm is even}\\
2^{n}-2^{(n-1)/2}, &\mbox{\rm if}\quad n\quad \mbox{\rm is odd}.\\
\end{array}
\right.
$$
\end{corollary}
\end{section}

\begin{section}{The Tower is Ordinary}\label{ordinary}

In this section we determine the $p$-rank of the second Garcia-Stichtenoth tower.
We do this by first computing the $p$-rank of its Galois closure.

The $p$-rank of an abelian variety (defined over a finite field of characteristic $p$)
is the $\mathbb{F}_p$-dimension
of the group of $p$-torsion points, considered over the algebraic closure.
The $p$-rank is invariant under isogenies, and lies between 0 and the dimension
of the abelian variety.  
The $p$-rank of $A\times B$ is the sum of the $p$-ranks of $A$ and $B$.
If the $p$-rank is maximal, i.e., equal to the dimension,
the abelian variety is said to be ordinary.
We will prove that $J_m$ is ordinary. 
We do this by showing that the $p$-rank is equal to the genus of $\tilde X_n$,
using the following formula for the genus which was proved in  \cite{Zaytsev1}:
\[
g(\tilde X_n)=[\tilde {\rm T}_n : {\rm T}_1] (p-p^{3-n}-p^{2-n})+1.
\]

We will use the Deuring-Shafaravich formula (for example see \cite{Crew}), which states that
if $E/F$ is a finite Galois extension of function fields in characteristic $p$,
and the Galois group is a $p$-group, then
\[
r_p(E)-1=[E:F] (r_p(F)-1)+\sum_P (e(P)-1)
\]
where $r_p(E)$ denotes the $p$-rank of $E$, and $e(P)$ denotes the ramification index.
When we talk about the $p$-rank of $E$, we mean the $p$-rank of the jacobian.

We apply this to the tower $[\tilde {\rm T}_n : {\rm T}_1]$, where is it known that the Galois group
is a $p$-group.

\begin{theorem}\label{galoisclosureordinary}
We have
\[
r_p(\tilde {\rm T}_n)=[\tilde {\rm T}_n : {\rm T}_1] (p-p^{3-n}-p^{2-n})+1.
\]
In particular, $\tilde {\rm T}_n$ is ordinary.
\end{theorem}

\begin{proof}
Let $d=[\tilde {\rm T}_n : {\rm T}_1]$.
We certainly have $p({\rm T}_1)=0$ since ${\rm T}_1$ is a projective line.
The computation of the ramification indices is done in Proposition 5.1 of  \cite{Zaytsev1}.
There is one point ($\infty$) with $d/p^{n-3}$ points lying over it, each having
ramification index $p^{n-3}$.
Additionally, there are another $p$ points each having $d/p^{n-1}$ points over them,
where each of these has ramification index $p^{n-1}$.
By the Deuring-Shafaravich formula we get
\[
r_p(\tilde {\rm T}_n)-1=d(-1)+\frac{d}{p^{n-3}} (p^{n-3}-1)+p\frac{d}{p^{n-1}} (p^{n-1}-1)
\]
which can be easily seen to equal $g(\tilde X_n)$
as given above.
\end{proof}

\begin{corollary}
The second Garcia-Stichtenoth tower is ordinary.
\end{corollary}

\begin{proof}
By Theorem \ref{galoisclosureordinary} we know that the Galois closure
of the tower is ordinary. 
However the Jacobian of the tower ${\rm T}_n$ is isogenous to a subvariety of
the Jacobian of $\tilde {\rm T}_n$.
Since $p$-rank is invariant under isogeny, the tower ${\rm T}_n$ must also 
have full $p$-rank and is therefore ordinary.
\end{proof}
\end{section}
\def\cprime{$'$} \def\cprime{$'$} \def\cprime{$'$}

\end{document}